\documentclass[a4paper,12pt]{amsart}
\usepackage{latexsym,amsmath,amssymb}

\newtheorem{thm}{Theorem}[section]
\newtheorem{cor}[thm]{Corollary}
\newtheorem{lem}[thm]{Lemma}
\newtheorem{prop}[thm]{Proposition}
\theoremstyle{definition}
\newtheorem{defn}[thm]{Definition}
\theoremstyle{remark}
\newtheorem{rem}[thm]{Remark}
\numberwithin{equation}{section}

\begin{document}
\title[]{\textsc{\textbf{ SPECTRAL RADIUS ALGEBRAS OF WCE OPERATORS}}}
\author{\sc\bf Y. Estaremi and M. R. Jabbarzadeh }
\address{\sc Y. Estaremi and M. R. Jabbarzadeh }
\email{yestaremi@pnu.ac.ir} \email{mjabbar@tabrizu.ac.ir}

\address{Department of Mathematics, Payame Noor University , P. O. Box: 19395-3697, Tehran,
Iran}
\address{Faculty of Mathematical
Sciences, University of Tabriz, P. O. Box: 5166615648, Tabriz,
Iran}

\thanks{}
\thanks{}
\date{}
\dedicatory{}

\begin{abstract}
In this paper, we investigate the spectral radius algebras related to the weighted conditional expectation operators on the Hilbert spaces $L^2(\mathcal{F})$. We give a large classes of operators on $L^2(\mathcal{F})$ that have the same spectral radius algebra. As a consequence we get that the spectral radius algebras of a weighted conditional expectation operator and its Aluthge transformation are equal.  Also, we obtain an ideal of the spectral radius algebra related to the rank one operators on the Hilbert space $\mathcal{H}$. Finally we get that the operator $T$ majorizes all closed range elements of the spectral radius algebra of $T$, when $T$ is a weighted conditional expectation operator on $L^2(\mathcal{F})$ or a rank one operator on the arbitrary Hilbert space $\mathcal{H}$.
\end{abstract}

\maketitle

\commby{}


\section{\textsc{Introduction}}
Let $(X, \mathcal{F}, \mu)$ be a complete $\sigma$-finite measure space. All sets and functions statements  are to be interpreted as holding up to sets of measure zero.
For a $\sigma$-subalgebra $\mathcal{A}$ of $\mathcal{F}$, the conditional expectation operator associated with $\mathcal{A}$ is the mapping $f\rightarrow E^{\mathcal{A}}f,$ defined for all non-negative $f$ as well as for all $f\in L^2(\mathcal{F})=L^2(X, \mathcal{F}, \mu)$, where $E^{\mathcal{A}}f$ is the unique $\mathcal{A}$-measurable function satisfying
$\int_{A}(E^{\mathcal{A}}f)d\mu=\int_{A}fd\mu$, for all $A\in \mathcal{A}$.
We will often write $E$ for $E^{\mathcal{A}}$.
The mapping
$E$ is a linear orthogonal projection from  $L^2(\mathcal{F})$ onto  $L^2(\mathcal{A})$.
For more details on the properties of $E$ see \cite{rao}.\\

We continue our investigation about the class of bounded linear operators on the $L^p$-spaces having the form $M_wEM_u$, where $E$ is the conditional expectation operator, $M_w$ and $M_u$ are (possibly unbounded) multiplication operators and it is called weighted conditional expectation operator. Our interest in operators of the form $M_wEM_u$ stems from the fact that such forms tend to appear often in the study of those operators related to conditional expectation. Weighted conditional expectation operators appeared in \cite{dou}, where it is shown that every contractive projection on certain $L^1$-spaces can be decomposed into an operator of the form $M_wEM_u$ and a nilpotent operator. For more strong results about weighted conditional expectation operators one can see \cite{dhd,g,lam,mo}. In these papers one can see that a large classes of operators are of the form of weighted conditional expectation operators.\\

Let $\mathcal{H}$ be a Hilbert space and $\mathcal{B}(\mathcal{H})$ be the algebra of all bounded linear operators on $\mathcal{H}$.
$\mathcal{N}(T)$ and
$\mathcal{R}(T)$ denote the null-space and range of an operator $T$,
respectively.
A closed subspace $\mathcal{M}$ of $\mathcal{H}$ is said to be invariant for an operator $T\in\mathcal{B}(\mathcal{H})$ if $T\mathcal{M}\subseteq \mathcal{M}$. The collection of all invariant subspaces of $T$ is a lattice and it is denoted by $Lat(T)$. If $\mathcal{M}$ is invariant for all operators commute with $T$, then it is called a hyperinvariant subspace for $T$. The description $Lat(T)$ is an open problem. Some author describes $Lat(T)$ in the special case of $T$. In \cite{lp}, Lambert and Petrovic introduced a modified version of a class of operator algebras that is called spectral radius algebras. Since a spectral radius algebra related to an operator $T\in\mathcal{B}(\mathcal{H})$ ($\mathcal{B}_{T}$) contains all operators that commute with $T$ ($\{T\}'$), then the invariant subspaces of $\mathcal{B}_{T}$ are hyperinvariant subspaces of $T$. In \cite{lp}, the authors established several sufficient conditions for $\mathcal{B}_{T}$ to have a nontrivial invariant subspace. When $T$ is compact the results of \cite{lp} generalizes the Lomonosov's theorem. In \cite{bls}, the authors demonstrated that for a subclasses of normal operators $\mathcal{B}_{T}$ has a nontrivial invariant subspace.
Spectral radius algebras for complex symmetric operators are
discussed in \cite{jke}.\\

In this paper we investigate the spectral radius algebras related to the weighted conditional expectation operators on the Hilbert spaces $L^2(\mathcal{F})$. We will show that there are lots of operators on $L^2(\mathcal{F})$ such as $T$ with $\mathcal{B}_T\neq \{T\}'$. In addition, we obtain an ideal of the spectral radius algebra related to the rank one operators on the Hilbert space $\mathcal{H}$. Finally we get that the operator $T$ majorizes all closed range elements of the spectral radius algebra of $T$, when $T$ is a weighted conditional expectation operator on $L^2(\mathcal{F})$ or a rank one operator on the arbitrary Hilbert space $\mathcal{H}$.

\section{ \textsc{spectral radius algebras }}

For notation and basic terminology concerning spectral radius algebras, we refer the reader to
\cite{blpw,lp}.\\

Let $\mathcal{H}$ be a Hilbert space, $T\in \mathcal{B}(\mathcal{H})$ and let $r(T)$ be the spectral radius of $T$. For $m\geq1$ we define
\begin{equation}\label{e1}
R_m(T)=R_m:=\left(\sum^{\infty}_{n=0}d^{2n}_mT^{\ast ^n}T^n\right)^{\frac{1}{2}},
\end{equation}
where $d_m=\frac{1}{1/m+r(T)}$. Since $d_m\uparrow 1/r(T)$, the sum in (\ref{e1}) is norm convergent
and for each $m$,  $R_m$ is well defined, positive and invertible. The spectral radius algebra $\mathcal{B}_T$ of $T$ consists of all operators $S\in \mathcal{B}(\mathcal{H})$ such that
$$\sup_{m\in \mathbb{N}}\|R_mSR^{-1}_m\|<\infty.$$
$\mathcal{B}_T$ is an algebra and it contains all operators commute with $T$.
Throughout this section we assume that
$w,u\in\mathcal{D}(E):=\{f\in
L^0(\mathcal{F}): E(|f|)\in L^0(\mathcal{A})\}$.
Now we recall the definition of weighted conditional expectation operators on $L^2(\mathcal{F})$.
\begin{defn}
Let $(X,\mathcal{F},\mu)$ be a $\sigma$-finite measure space and let $\mathcal{A}$ be a
$\sigma$-subalgebra of $\mathcal{F}$ such that $(X,\mathcal{A},\mu_{\mathcal{A}})$ is
also $\sigma$-finite. Let $E$ be the conditional
expectation operator relative to $\mathcal{A}$. If $u, w \in L^0(\mathcal{F})$,
the spaces of $\mathcal{F}$-measurable functions on $X$, such that $uf$ is conditionable
\cite{her} and $wE(uf)\in L^{2}(\mathcal{F})$ for all $f\in \mathcal{D}\subseteq L^{2}(\mathcal{F})$,
where $\mathcal{D}$ is a linear subspace, then the corresponding weighted conditional
expectation (or briefly WCE) operator is the linear transformation
$M_wEM_u:\mathcal{D}\rightarrow L^{2}(\mathcal{F})$ defined by $f\rightarrow wE(uf)$.
\end{defn}
As was proved in \cite{ej} we have an equivalent condition for boundedness of the weighted conditional expectation operators $M_wEM_u$ on $L^2(\mathcal{F})$ as the next theorem.
\begin{thm}\label{t2}
The operator $T=M_{w}EM_{u}:
L^2(\mathcal{F})\rightarrow L^2(\mathcal{F})$ is bounded if and only if
$(E(|w|^2)^{\frac{1}{2}}) (E(|u|^2)^{\frac{1}{2}})\in
L^{\infty}(\mathcal{A})$, in this case
$\|T\|=\|(E(|w|^2)^{\frac{1}{2}})
(E(|u|^2)^{\frac{1}{2}})\|_{\infty}$.
\end{thm}

Let $T=M_wEM_u$ be a bounded operator on $L^2(\mathcal{F})$. Direct computations shows that for every $n\in \mathbb{N}$ (natural numbers) we have
\begin{align*}
T^nf&=(E(uw))^{n-1}wE(uf);\\
T^{\ast ^n}f&=(\overline{E(uw)})^{n-1}\bar{u}E(\bar{w}f).
\end{align*}

Since $R_m=R_m(M_wEM_u)$ is
positive and invertible operator, we obtain
$$R_m=\left(I+M_{(E(|w|^2)\sum^{\infty}_{n=1}d^{2n}_m|E(uw)|^{2(n-1)})}M_{\bar{u}}EM_u\right)^{\frac{1}{2}}.$$
It is easy to see that the following equality holds almost every where on $X$.
$$\sum^{\infty}_{n=1}d^{2n}_m|E(uw)|^{2(n-1)}=\frac{d^{2}_m}{1-d^2_m|E(uw)|^2}.$$
If we set
$$v_m=\frac{d^{2}_mE(|w|^2)}{1-d^2_m|E(uw)|^2},$$
then we have
$$R_m=\left(I+M_{v_m\bar{u}}EM_u\right)^{\frac{1}{2}}.$$
By an elementary technical method we can compute the inverse of $R_m$ as follow:
$$R^{-1}_m=\left( I+M_{\frac{v_m\bar{u}}{v_mE(|u|^2)-1}}EM_{u} \right)^{\frac{1}{2}}.$$

Here we recall a fundamental lemma in operator theory.
\begin{lem}\label{l1} Let $T$ be a bounded operator on the Hilbert space $\mathcal{H}$ and $\lambda\geq0$. Then we have
$$\|\lambda I+T^{\ast}T\|=\lambda+\|T^{\ast}T\|=\lambda+\|T\|^2.$$
Specially, if $T$ is a positive operator, then
$\|\lambda I+T\|=\lambda+\|T\|$.
\end{lem}
\begin{proof} It is an easy exercise.
\end{proof}

From now on, we assume that $E(|u|^2)\in L^{\infty}(\mathcal{A})$. Now we characterize the spectral radius algebra corresponding to the WCE operator $M_wEM_u$ in the next theorem.
\begin{thm}\label{t1}
Let $S\in \mathcal{B}(L^2(\mathcal{F}))$. Then $S\in \mathcal{B}_{M_wEM_u}$ if and only if $\mathcal{N}(EM_u)$ is invariant under $S$.
\end{thm}
\begin{proof}
Since $R_m$ and $R^{-1}_m$ are positive operators and $(EM_u)^{\ast}=M_{\bar{u}}E$, then by  Lemma \ref{l1} and Theorem \ref{t2} we have
$$\|R_m\|^2=\|R^2_m\|=1+\|E(|u|^2)v_m\|_{\infty}$$
and
$$\|R^{-1}_m\|^2=\|R^{-2}_m\|=1+\|\frac{E(|u|^2)v_m}{v_mE(|u|^2)-1}\|_{\infty}.$$
If we decompose  $L^2(\mathcal{F})$ as a direct sum $\mathcal{H}_1\oplus \mathcal{H}_2$, in which
 $$\mathcal{H}_2=\mathcal{N}(EM_u)=\{f\in L^2(\mathcal{F}): E(uf)=0\}$$
and
 $$\mathcal{H}_1=\mathcal{H}^{\perp}_2=\overline{\bar{u}L^2(\mathcal{A})},$$
 then the corresponding block matrix of $R_m$ is
\begin{center}
$ R_m=\left(
   \begin{array}{cc}
     M_{(q_m)^{\frac{1}{2}}} & 0 \\
     0 & I \\ \ \ \ \ \
   \end{array}
 \right)
$
and
$\  R^{-1}_m=\left(
   \begin{array}{cc}
     M_{(q_m)^{\frac{-1}{2}}} & 0 \\
     0 & I \\
   \end{array}
 \right)
$,
\end{center}
where $q_m=1+v_mE(|u|^2)$. Notice that for $m>m'$ we have $q_m\geq q_{m'}$ and $\|q_m\|_{\infty}\rightarrow \infty$ as $m\rightarrow \infty$. If $S\in \mathcal{B}(L^2(\mathcal{F}))$ say
$S=\left(
  \begin{array}{cc}
    X & Y \\
    Z & W \\
  \end{array}
\right)$, the block matrix with respect to the decomposition $\mathcal{H}_1\oplus \mathcal{H}_2$, then
\begin{align*}
R_mSR^{-1}_m&=\left(
   \begin{array}{cc}
     M_{(q_m)^{\frac{1}{2}}} & 0 \\
     0 & I \\
   \end{array}
 \right)\left(
  \begin{array}{cc}
    X & Y \\
    Z & W \\
  \end{array}
\right)\left(
   \begin{array}{cc}
     M_{(q_m)^{\frac{-1}{2}}} & 0 \\
     0 & I \\
   \end{array}
 \right)\\
 &=
\left(
   \begin{array}{cc}
     M_{(q_m)^{\frac{1}{2}}}XM_{(q_m)^{\frac{-1}{2}}} & M_{(q_m)^{\frac{1}{2}}}Y  \\
     ZM_{(q_m)^{\frac{-1}{2}}} & W\\
   \end{array}
 \right).
 \end{align*}
 Since $\|M_{(q_m)^{\frac{1}{2}}}XM_{(q_m)^{\frac{-1}{2}}}\|\leq\|X\|$, then we get that $\sup_m\|R_mSR^{-1}_m\|<\infty$ if and only if $\sup_m\|M_{(q_m)^{\frac{1}{2}}}Y\|<\infty$. Direct computations shows that $\sup_m\|M_{(q_m)^{\frac{1}{2}}}Y\|<\infty$ if and only if $Y=0$. This means that $\mathcal{H}_2$ is an invariant subspace for $S$.
\end{proof}

Therefore by Theorem \ref{t1} we get that there are many different operators that have the same spectral radius algebra.

\begin{cor} Let $w, w',u\in \mathcal{D}(E)$. If $M_wEM_u$ and $M_{w'}EM_u$ are bounded operator on the Hilbert space $L^2(\mathcal{F})$, then $\mathcal{B}_{M_{w'}EM_u}=\mathcal{B}_{M_wEM_u}$.
\end{cor}

Also in the next corollary we have a sufficient condition for $\mathcal{B}_{M_wEM_u}$ to be equal to $\mathcal{B}(L^2(\mathcal{F}))$.

\begin{cor} If $\mathcal{N}(EM_u)=\{0\}$, then $\mathcal{B}_{M_wEM_u}=\mathcal{B}(L^2(\mathcal{F}))$.
\end{cor}
In the next Proposition we find some special elements of $\mathcal{B}_{M_wEM_u}$.
\begin{prop}\label{p111} If $a\in L^0(\mathcal{A})$ such that $a\geq0$ and $M_{a\bar{u}}EM_u\in \mathcal{B}(L^2(\mathcal{F}))$, then $M_{a\bar{u}}EM_u\in\mathcal{B}_{M_wEM_u}$.
\end{prop}
\begin{proof} Since $R_m=\left(I+M_{v_m\bar{u}}EM_u\right)^{\frac{1}{2}}$ and $v_m=\frac{d^{2}_mE(|w|^2)}{1-d^2_m|E(uw)|^2}$ is an $\mathcal{A}$-measurable function, it holds that $R_m M_{a\bar{u}}EM_u=M_{a\bar{u}}EM_uR_m$. Therefore we have
$\|R_m M_{a\bar{u}}EM_u R^{-1}_m\|=\|M_{a\bar{u}}EM_u\|$, and so we get that $M_{a\bar{u}}EM_u\in\mathcal{B}_{M_wEM_u}$.
\end{proof}

Every operator $T$ on a Hilbert
space $\mathcal{H}$ can be decomposed into $T = U|T|$ with a
partial isometry $U$, where $|T| = (T^*T)^{\frac{1}{2}}$. $U$ is
determined uniquely by the kernel condition $\mathcal{N}(U) =
\mathcal{N}(|T|)$. Then this decomposition is called the polar
decomposition. The Aluthge transformation $\widetilde{T}$ of the
operator $T$ is defined by
$\widetilde{T}=|T|^{\frac{1}{2}}U|T|^{\frac{1}{2}}$.
Here we recall that the Aluthge transformation of
$T=M_wEM_u$ is
$$\widetilde{T}(f)=\frac{\chi_{z_1}E(uw)}{E(|u|^{2})}\bar{u}E(uf), \ \ \ \ \ \  \ \ \ \  \ \  \ f\in L^{2}(\mathcal{F}),$$
in which $z_1=z(E(|u|^2))$ (see \cite{ej}). Thus $\widetilde{T}=M_{w'}EM_{u'}$ where $w'=\frac{E(uw)\bar{u}\chi_{z_1}}{E(|u|^2)}$ and $u'=u$. We recall that $r(M_wEM_u)=\|E(uw)\|_{\infty}$ (see \cite{e}). Direct computations shows that $E(u'w')=E(uw)$. Hence $r(T)=r(\widetilde{T})$. Hence by using  Proposition \ref{p111} we have the next corollary.
\begin{cor} If $w$ and $u$ are positive measurable functions, then $\widetilde{T}\in \mathcal{B}_{T}$ where $T=M_wEM_u$.
\end{cor}
By the proof of Proposition \ref{p111} we get that the commutant of $M_wEM_u$ (in symbol $\{M_wEM_u\}'$) is a proper subset of $\mathcal{B}_{M_wEM_u}$ when $w, u$ are positive and $w\neq u$. In the next theorem we get that $\mathcal{B}_{T}=\mathcal{B}_{\widetilde{T}}$ when $T=M_wEM_u$ and $w,u\geq0$.
\begin{cor} If $T=M_wEM_u$ and $w,u\geq0$, then $\mathcal{B}_{T}=\mathcal{B}_{\widetilde{T}}$.
\end{cor}

Recall that for $f, g\in L^2(\mathcal{F})$ we can define a rank one operator $f\otimes g$ on $L^2(\mathcal{F})$ by the action $(f\otimes g)(h)=\langle h, g\rangle f$ for every $h\in L^2(\mathcal{F})$, in which $\langle\ ,\ \rangle$ is the inner product of the Hilbert space $L^2(\mathcal{F})$.
In the next proposition we give some conditions under which a rank one operator belongs to the $\mathcal{B}_{M_wEM_u}$.
\begin{prop} If $T=M_wEM_u$ and $f,g\in L^2(\mathcal{F})$, then $f\otimes g\in \mathcal{B}_T$ if and only if
$$\sup_m\|\alpha_m^{\frac{1}{2}}E(ug)\|^2\|f\|^2+\|v_m^{\frac{1}{2}}E(uf)\|^2(\|g\|^2+\|\alpha_m^{\frac{1}{2}}E(ug)\|^2)<\infty,$$
where $\alpha_m=\frac{v_m}{v_mE(|u|^2)-1}$.
\end{prop}
\begin{proof} By using the properties of inner product we have
$$\|R_m{f}\|^2=\|f\|^2+\|v_m^{\frac{1}{2}}E(uf)\|^2$$
and
$$\|R^{-1}_m{f}\|^2=\|g\|^2+\|\alpha_m^{\frac{1}{2}}E(ug)\|^2.$$
Now, the desired conclusion follows by \cite[Lemma 3.9]{lp}.
\end{proof}

By using some results of \cite{lp} we get that the conditional expectation corresponding to $\sigma$- subalgebra $\mathcal{A}\subseteq\mathcal{B}$ are in $\mathcal{B}_{E^{\mathcal{A}}M_u}$ as we mentioned in the next remark.
\begin{rem} Let $T=EM_u\in \mathcal{B}(L^2(\mathcal{F}))$, $u\in L^{\infty}(\mathcal{A})$ and let $\mathcal{A}, \mathcal{B}$ be $\sigma$-subalgebras of $\mathcal{F}$ such that $\mathcal{A}\subseteq \mathcal{B}$. If $E=E^{\mathcal{A}}$ and $S$ is an operator for which $TS=E^{\mathcal{B}}ST$, then $S\in \mathcal{B}_{T}$.
\end{rem}
\begin{proof}
It is not hard to see that $EM_uE^{\mathcal{B}}=E^{\mathcal{B}}EM_u$. Since $E^{\mathcal{B}}$ is a projection on $L^2(\mathcal{F})$, then it is power bounded. Therefore, by \cite[Proposition 2.3]{lp} we get that $S\in \mathcal{B}_{T}$.
\end{proof}

\begin{cor}\label{c11} If $T=M_wEM_u\in \mathcal{B}(L^2(\mathcal{F}))$ and $a\in L^{\infty}(\mathcal{A})$, then $M_a\in \mathcal{B}_{T}$.
\end{cor}

Let $\mathcal{H}$ be a Hilbert space and $T\in \mathcal{B}(\mathcal{H})$. Here we recall the definition of $Q_T$, that is defined in \cite{lp}, as follows:
$$Q_T=\{S\in \mathcal{B}(\mathcal{H}):\|R_mSR^{-1}_m\|\rightarrow0\}.$$
In the next theorem we illustrate $Q_T$ when $T=M_wEM_u\in \mathcal{B}(L^2(\mathcal{F}))$.

\begin{thm}\label{t3}
Let $T=M_wEM_u$ and $S\in\mathcal{B}(L^2(\mathcal{F}))$. Then $S\in Q_T$ if and only if   $\mathcal{N}(EM_u)$ is invariant under $S$ and $\mathcal{N}(EM_u)\subseteq \mathcal{N}(S)$.
\end{thm}
\begin{proof} Let
$S=\left(
  \begin{array}{cc}
    X & Y \\
    Z & W \\
  \end{array}
\right)$, the block matrix with respect to the decomposition $\mathcal{H}_1\oplus \mathcal{H}_2$, in which $\mathcal{H}_2=\mathcal{N}(EM_u)$ and $\mathcal{H}_1=\mathcal{H}^{\perp}_2$. So similar to the proof of Theorem \ref{t1} we have
$$R_mSR^{-1}_m=
\left(
   \begin{array}{cc}
     M_{(q_m)^{\frac{1}{2}}}XM_{(q_m)^{\frac{-1}{2}}} & M_{(q_m)^{\frac{1}{2}}}Y  \\
     ZM_{(q_m)^{\frac{-1}{2}}} & W\\
   \end{array}
 \right).$$

 Hence $\|W\|\leq \|R_mSR^{-1}_m\|$. Since $S\in Q_T$, then $Y=0$ and $\|W\|=0$. This means that $\mathcal{H}_2$ is invariant under $S$ and $PSP=0$ in which $P=P_{\mathcal{H}_2}$. Therefore $SP=PSP=0$, and so $\mathcal{H}_2\subseteq \mathcal{N}(S)$. Conversely, if $\mathcal{N}(EM_u)$ is invariant under $S$ and $\mathcal{N}(EM_u)\subseteq \mathcal{N}(S)$, then we get that $W=Y=0$ and
 $$R_mSR^{-1}_m=
\left(
   \begin{array}{cc}
     M_{(q_m)^{\frac{1}{2}}}XM_{(q_m)^{\frac{-1}{2}}} & 0  \\
     ZM_{(q_m)^{\frac{-1}{2}}} & 0\\
   \end{array}
 \right).$$
 Hence
 $$\|R_mSR^{-1}_m\|\leq \|M_{(q_m)^{\frac{1}{2}}}XM_{(q_m)^{\frac{-1}{2}}}\|+\|ZM_{(q_m)^{\frac{-1}{2}}}\|.$$
 Since $\|M_{(q_m)^{\frac{-1}{2}}}\|=\|(q_m)^{\frac{-1}{2}}\|_{\infty}\rightarrow 0$, then $\|R_mSR^{-1}_m\|\rightarrow0$ when $m\rightarrow\infty$. This completes the proof.
\end{proof}
Now by using \cite[Theorem 2.6]{lp} and some information about WCE operators we have an equivalent condition for the spectral radius algebra of a WCE operator to be equal to $\mathcal{B}(L^2(\mathcal{F}))$.
\begin{prop}\label{p21}
If $T=M_wEM_u$, then $\mathcal{B}_{T}=\mathcal{B}(L^2(\mathcal{F}))$ if and only if $$\sup_m(\|E(|u|^2)v_m\|_{\infty}+\|\frac{E(|u|^2)v_m}{v_mE(|u|^2)-1}\|_{\infty}(1+\|E(|u|^2)v_m\|_{\infty}))<\infty,$$
where $v_m=\frac{d^{2}_mE(|w|^2)}{1-d^2_m|E(uw)|^2}$.
\end{prop}
\begin{proof}
It is a direct consequence of \cite[Theorem 2.6]{lp} and some information of the proof of  Theorem \ref{t1}.
\end{proof}
By using  Proposition \ref{p21} and some results of \cite{bls} we have an equivalent condition for the WCE operator $M_wEM_u$ to be a constant multiple of an isometry.
\begin{thm} If $T=M_wEM_u$ is a bounded operator on the Hilbert space $L^2(\mathcal{F})$, then $T$ is a constant multiple of an isometry if and only if
$$\sup_m(\|E(|u|^2)v_m\|_{\infty}+\|\frac{E(|u|^2)v_m}{v_mE(|u|^2)-1}\|_{\infty}(1+\|E(|u|^2)v_m\|_{\infty}))<\infty,$$
where $v_m=\frac{d^{2}_mE(|w|^2)}{1-d^2_m|E(uw)|^2}$.
\end{thm}
\begin{proof}
It is a direct consequence of  \cite[Theorem 2.7 ]{bls} and  Proposition \ref{p21}.
\end{proof}
Now in the next theorem we obtain some sufficient conditions for $\mathcal{B}_{M_wEM_u}$ to a nontrivial invariant subspace.
\begin{thm} If the measure space $(X, \mathcal{A}, \mu)$ is not a non-atomic measure space and $E(uw)=0$, then $\mathcal{B}_{M_wEM_u}$ has a nontrivial invariant subspace.
\end{thm}
\begin{proof}
Since $E(uw)=0$ then $M_wEM_u$ is quasinilpotent. Also since the $\sigma$-algebra $\mathcal{A}$ has at least one atom, then we have a compact multiplication operator $M_a$ for some $a\in L^{\infty}(\mathcal{A})$. Hence by Corollary \ref{c11} we have $M_a\in \mathcal{B}_{M_wEM_u}$. Moreover by using \cite[Lemma 3.1]{lp} we get that $M_wEM_u\in Q_{M_wEM_u}$. Therefore by \cite[Theorem 3.4]{lp} we get the proof.
\end{proof}
Here we give a remark on  \cite[Proposition 2.8]{lp} as follows:
\begin{rem} For the unit vectors $u, v, w$ of the Hilbert space $\mathcal{H}$ we have $\mathcal{B}_{u\otimes w}=\mathcal{B}_{v\otimes w}$.
\end{rem}
In the next theorem we describe $Q_{u\otimes v}$ for a rank one operator $u\otimes v$ in which $u, v$ are in the Hilbert space $\mathcal{H}$.
\begin{thm}\label{p10} Let $\mathcal{H}$ be a Hilbert space and $S\in\mathcal{B}(\mathcal{H})$.  If $u, v\in \mathcal{H}$, then $S\in Q_{u\otimes v}$ if and only if $S=(I-P)SP$, where $P=P_{\mathcal{H}_1}$ and $\mathcal{H}_1$ is the one-dimensional space spanned by $v$.
\end{thm}
\begin{proof} As was computed in  \cite[Proposition 2.8]{lp} we have
$$R^2_m=I+\frac{d^2_m}{1-d^2_mr^2}v\otimes v,$$
in which $r=r(u\otimes v)=|\langle u, v\rangle|$. Let $\lambda_m=\sqrt{1+\frac{d^2_m}{1-d^2_mr^2}}$. If $\mathcal{H}_1$ is the one-dimensional space spanned by $v$ and $\mathcal{H}_2=\mathcal{H}^{\perp}_1$. For $S\in\mathcal{B}(\mathcal{H})$, we have the corresponding block matrix of $R_m$, $R^{-1}_m$ and $S$ with respect to the decomposition $\mathcal{H}= \mathcal{H}_1\oplus \mathcal{H}_2$ as follows:
\begin{center}
$ R_m=\left(
   \begin{array}{cc}
     M_{\lambda_m} & 0 \\
     0 & I \\
   \end{array}
 \right)
$, \ \
$ R^{-1}_m=\left(
   \begin{array}{cc}
     M_{\frac{1}{\lambda_m}} & 0 \\
     0 & I \\
   \end{array}
 \right)
$
\end{center}
and
$$S=\left(
  \begin{array}{cc}
    PSP & PS(I-P) \\
    (I-P)SP & (I-P)S(I-P) \\
  \end{array}
\right).$$
Therefore, we have
$$R_mSR^{-1}_m=
\left(
   \begin{array}{cc}
     PSP & M_{\lambda_m}PS(I-P)  \\
     M_{\frac{1}{\lambda_m}}(I-P)SP & (I-P)S(I-P)\\
   \end{array}
 \right).$$

 If $S\in Q_{u\otimes v}$, then $S\in \mathcal{B}_{u\otimes v}$. Hence by \cite[Proposition 2.8]{lp} we obtain $PS(I-P)=0$. Since $S\in Q_{u\otimes v}$, $\|PSP\|\leq\|R_mSR^{-1}_m\|$ and $\|(I-P)S(I-P)\|\leq\|R_mSR^{-1}_m\|$, then $\|PSP\|=0$ and $\|(I-P)S(I-P)\|=0$. Hence $PSP=0$ and $(I-P)S(I-P)=0$. Thus $$S=\left(
  \begin{array}{cc}
    0 & 0 \\
    (I-P)SP & 0\\
  \end{array}
\right)=(I-P)SP.$$ Conversely, If $S=(I-P)SP$, then
$$\|R_mSR^{-1}_m\|=\left\|
\left(
   \begin{array}{cc}
     0 & 0  \\
     M_{\frac{1}{\lambda_m}}(I-P)SP & 0\\
   \end{array}
 \right)\right\|=\|M_{\frac{1}{\lambda_m}}(I-P)SP\|.$$
 Since $\|M_{\frac{1}{\lambda_m}}(I-P)SP\|\rightarrow 0$, then $\|R_mSR^{-1}_m\|\rightarrow0$ as $m\rightarrow \infty$. So $S\in Q_{u\otimes v}$.
\end{proof}

Let $X, Y, Z$ be Banach spaces. Assume that $T\in \mathcal{B}(X, Y)$ and $S\in \mathcal{B}(X, Z)$. Then $T$ majorizes $S$ if there exists $M>0$ such that
$$\|Sx\|\leq M\|Tx\|$$
for all $x\in X$ (see \cite{bb}). Here we recall a result of \cite{bb} that gives us an equivalent condition for a closed range operator to majorize another bounded operator.
\begin{rem}\cite[Proposition 4]{bb}\label{r1}
Let $X$ be Banach spaces and $T, S\in \mathcal{B}(X)$ with $\mathcal{R}(T)$ closed. Then $T$ majorizes $S$ if and only if $\mathcal{N}(T)\subseteq \mathcal{N}(S)$.
\end{rem}

Now we recall an assertion about closed range weighted conditional expectation operators.

\begin{prop}\cite[Theorem 2.1]{es}\label{p4}
If $z(E(u))=z(E(|u|^{2}))$ and for some
$\delta>0$, $E(u)\geq\delta$  on
$z(E(|u|^{2}))$, then the
operator $EM_u$ has closed range on $L^{2}(\mathcal{F})$.
\end{prop}

\begin{prop} Let $T=M_wEM_u$ and $u\geq0$. If $S\in Q_{T}$ and $E(u)\geq\delta$, then $EM_u$ majorizes $S$.
\end{prop}

\begin{proof} Since $u\geq0$, then $z(E(u))=z(E(|u|^{2}))$. Hence by the Remark \ref{r1},  Theorem \ref{t3} and  Proposition \ref {p4} we get the proof.
\end{proof}
Finally, since the rank one operator $x\otimes y$ has closed range, the we can obtain the next proposition.

\begin{prop}
Let $x,y\in \mathcal{H}$. If $T\in Q_{x\otimes y}$, then $x\otimes y$ majorizes $T$.
\end{prop}

\begin{proof} If $T\in Q_{x\otimes y}$, then by the proof of Theorem \ref{p10} we have $\mathcal{H}_2=\mathcal{N}(x\otimes y)$ and $\mathcal{N}(x\otimes y)\subseteq \mathcal{N}(T)$. Since $x\otimes y$ has closed range, then by the Remark \ref{r1} we conclude that $x\otimes y$ majorizes $T$.
\end{proof}

\end{document}